\documentclass{amsart}
\usepackage{amssymb}
\usepackage{amscd}
\usepackage{amsmath}
\usepackage{enumerate}
\usepackage[dvips]{graphicx}
\newtheorem{theo}{Theorem}

\newtheorem{cor}[theo]{Corollary}

\newtheorem{example}[theo]{Example}

\newcommand{\CC}{{\mathbb{C}}}

\newcommand{\PP}{{\mathbb{P}}}

\newcommand{\RR}{{\mathbb{R}}}

\newcommand{\SSS}{{\mathbb{S}}}
\newcommand{\ZZ}{{\mathbb{Z}}}

\newcommand{\calO}{{\mathcal{O}}}

\begin{document}
\title[Complements and Milnor fibres]{On homotopy types of complements of analytic sets and Milnor fibres}
\author{Javier Fern\'andez de Bobadilla}
\address{ICMAT. CSIC-Complutense-Aut\'onoma-Carlos III}
\email{javier@mat.csic.es}
\thanks{Research partially supported by the ERC Starting Grant project TGASS and by Spanish Contract MTM2007-67908-C02-02. The author thanks to the Faculty de Ciencias Matem\'aticas of the Universidad Complutense de Madrid for excellent working conditions.}
\dedicatory{Dedicated to A. Libgober in the ocassion of his 60-th birthday.}
\date{10-7-2008}
\begin{abstract}
We prove that for any germ of complex analytic set in $\CC^n$ there exists a hypersurface singularity whose Milnor fibration has trivial geometric monodromy and fibre
homotopic to the complement of the germ of complex analytic set. As an application we show an example of a quasi-homogeneous hypersurface singularity,
with trivial geometric monodromy and simply connected and non-formal Milnor fibre.
\end{abstract}
\maketitle
\section{Introduction}

The Milnor fibre of an isolated hypersurface singularity has the homotopy type of a bouquet of spheres of real dimension
equal to the complex dimension of the hypersurface. For non-isolated singularities there are very few general results
on the homotopy type of the Milnor fibre. Most notable and classical is Kato-Matsumoto bound which states that the 
Milnor fibre is $s-1$ connected if $S$ is the codimension of the singular set in the hypersurface.
For some special classes of singularities Siersma, Zaharia, T. De Jong, Nemethi and Shubladze\cite{Si},~\cite{Za},~\cite{dJ},~\cite{Ne},~\cite{Sh} have proved that 
the Milnor fibre is a bouquet of spheres of different dimensions. We show here that the class of homotopy types that
can be realised by the Milnor fibre of a hypersurface singularity is very rich: any homotopy type of a complement 
of germ of analytic space is in this class (Theorem~\ref{complementos}).
As a consequence complements of global quasi-homogeneous algebraic sets in 
$\CC$ are in this class, and the product of the circle and the complement of any projective algebraic set containing a
hyperplane are also in this class (Corollaries~\ref{quasihomogeneous} and~\ref{projective}).

Since the foundational work of Deligne, Griffiths, Morgan and Sullivan\cite{Su},~\cite{DGMS},~\cite{M}, 
and later Navarro-Aznar~\cite{NA},
it has been realised that the rational homotopy
type of compact K\"ahler manifolds and algebraic varieties is quite special and can be described in terms of the 
cohomology ring of the variety, the cohomology rings of certain subspaces involved in its compactification or
resolution of singularities, and natural mappings among them. When the rational homotopy type is determined only 
by the cohomology ring of the variety we say that the manifold is formal. Although a quasi-projective variety is not
formal in general, many classes of varieties, like compact K\"ahler manifolds~\cite{DGMS} or complements of hyperplane arrangements~\cite{Br}, 
are formal. In the last years understanding formality in the context of 
algebraic geometry and singularity theory has been subject of intense research (\cite{CM},\cite{DS},\cite{DPS1},\cite{DPS2},\cite{DPS3},\cite{PS1},\cite{PS2},\cite{Zu}).

The local Milnor fibre of a hypersurface singularity has very restrictive algebro-topological properties in common with algebraic varieties (e.g. polarised Mixed-Hodge structure in its cohomology). 
It is therefore natural to study its rational homotopy type in a similar fashion than in the case of algebraic
varieties. Very recently H. Zuber~\cite{Zu} has given the first
example of non-formal Milnor fibre of a quasihomogeneous polynomial , answering Question~5.5~of~\cite{PS1} (the Milnor fibre of a quasihomogeneous polynomial). 
Such Milnor fibre is non simply connected. It is clear from the literature than formality has different flavours in the simply connected and 
non-simply connected case. As an application of our main Theorem and results of Denham and Suciu on moment-angle complexes 
we find, up to our knowledge, the first family of quasihomogeneous polynomials with Milnor fibres non-formal and simply connected. Moment-angle manifolds and intersections of quadrics are also studied by S. lopez de Medrano and S. Gitler 
in~\cite{LdM}~and~\cite{GLdM}. The study  of this kind of manifolds partly is motivated by the theory of dynamical systems.

I would like to thank A. Suciu for a conversation in the 60th birthday
of A. Libgober celebration conference in Jaca in which tha application given here to formality arose, and also to correct a mistake in the computation of the homology of
the Milnor fibre in the last example of the paper. I thank also S. Lopez de Medrano for an interesting conversation in Oaxaca'a meeting of the Spanish and Mexican Mathematical Societies. I dedicate the paper to A. Libgober.

\section{Main Theorem}

Let $I=(f_1,\ldots,f_r)$ be any ideal in $\calO_{\CC^n,O}$. Let $(Z,O)$ the germ of complex analytic set defined by it. By the conical structure of analytic sets, the
topological type of the complement $C(I):=B_{\epsilon_1}\setminus Z$ is independent of $\epsilon$, if $\epsilon>0$ is small enough (the sphere of radius $\epsilon$ centered
at the origin is denoted by $B_{\epsilon}$).

\begin{theo}
\label{complementos}
Let $I=(f_1,\ldots,f_r)$ be any ideal in $\calO_{\CC^n,O}$. The Milnor fibre of the function germ
\[F_I:(\CC^n+r,O)\to\CC\]
defined by 
\[f(x_1,\cdots,x_n,y_1,\cdots,y_r):=\sum_{i=1}^ry_if(x_1,\cdots,x_r)\]
has the homotopy type of $C(I)$. Its geometric monodromy is trivial.
\end{theo}
\begin{proof}
The key is the linearity of $f_I$ in the $y_i$'s. The geometric monodromy statement is obvious: the transformation
$\varphi_{\theta}(x,y):=(x,e^{2\pi i\theta}y)$ takes the fibre over $\delta$ to the fibre over $e^{2\pi i\theta}\delta$ and is the identity for $\theta=1$.

It is well known that fhe homotopy type of the Milnor fibre is the same taking spheres or polydiscs as neighbourhoods of the origin.
The system of neighbourhoods of the origin that we will use is
\[N_{\epsilon,\epsilon}:=\{(x,y):||x||\leq\epsilon, ||y_i||\leq\epsilon, 1\leq i\leq r\}.\]

Take $\epsilon_1$, $\epsilon_2$ and $\delta$ small enough so that $F_I:=f_I^{-1}(\delta)\cap N_{\epsilon_1,\epsilon_2}$
is the Milnor fibre of $f_I$. Consider the projection
\[\pi:N_{\epsilon,\epsilon}\to B_{\epsilon}\]
defined by $\pi(x,y):=x$. Define $\xi:=\delta/\epsilon$. I claim that the image $\pi(F_I)$ is equal to the set
\[C(I,\xi):=B_\epsilon\setminus Z(I,\xi),\]
where
\[Z(I,\xi):=\{x\in B_\epsilon: \sum_{i=1}^r||f_i(x)||<\xi\}.\]

Indeed, given $x\in Z(I,\xi)$ and $y_1,\ldots,y_r$ such that $||y_i||\leq\epsilon$ for any $i$ we have
\[||f_I(x,y)||\leq \sum_{i=1}^r||y_i||||f_i(x)||\leq\epsilon\sum_{i=1}^r||f_i(x)||\leq\epsilon\xi<\delta.\]
Therefore $f_I(x,y)\neq\delta$. If $x$ belongs to $C(I,\xi)$ then
$\sum_{i=1}^r||f_i(x)||\geq\xi$. Hence we have the inequality
\[\eta:=\delta/\sum_{i=1}^r||f_i(x)||\leq\epsilon.\]
For any $1\leq j\leq r$ choose an argument $\theta_j$ such that
$e^{2\pi i\theta_j}f_j(x)$ is real and non-negative. Then 
$(x,\eta e^{2\pi i\theta_1},\ldots,\eta e^{2\pi i\theta_r})$ belongs to $F_I$. This proves the claim.

For any $x\in C(I,\xi)$ the fibre $\pi^{-1}(x)\cap F_I$ is the intersection of the polydisc defined by the conditions
$||y_i||\leq \epsilon$ with the linear subspace defined by $F_I(x,y)=\delta$
(the equation is linear in the $y$-coordinates). Such intersection is always convex, and hence contractible.
From this point it is easy to see that the restriction 
\[\pi:F_I\to C(I,\xi)\]
is a homotopy-fibration with contractible fibre, and therefore a homotopy equivalence.

If $\xi$ approaches $0$ as $\delta$ does. Thus, for $\delta$ small enough $Z(I,\xi)$ is an algebraic neighbourhood of 
$Z$, and admits it as a deformation retract. Therefore $C(I,\xi)$ is homotopy equivalent to $C(I)$.
\end{proof}

\begin{cor}
\label{quasihomogeneous}
If $I$ is a quasi-homogeneous ideal $f_I$ is also quasihomogeneous, and the global Milnor fibre is homotopic to the 
complement of the whole $\CC^n$ minus the algebraic set $V(I)$.
\end{cor}

\begin{cor}
\label{projective}
Let $Z\subset\PP^n$ be a closed algebraic set containg a hyperplane. Then $Z\times\SSS^1$ is homotopic to the Milnor
fibre of a homoheneous polynomial with trivial geometric monodromy.
\end{cor}
\begin{proof}
The complement of the cone in $\CC^{n+1}$ associated to $Z$ fibre over $\PP^n\setminus Z$ 
with fibre $\CC^*$ (homotopic to $\SSS^1$). The fibration
is trivial since $Z$ contains a hyperplane, and the fibration of $\CC^{n+1}\setminus{O}$ over $\PP^n$ trivialises in
the contractible complement of any hyperplane.
\end{proof}

\section{An application}

\begin{theo}
\label{nonformality}
There are quasi-homogeneous polynomials with non-formal simply connected Milnor fibre and geometric monodormy
equal to the identity.
\end{theo}
\begin{proof}
After Corollary~\ref{quasihomogeneous} it is enough to find a non-formal simply connected complement of algebraic set.
In~\cite{DS}, Section~9, it is given a family of arrangements of linear subspaces of complex codimension at least $2$, 
which are all intersections of coordinate hyperplanes, and such that their complement are non-formal 
(they have non-trivial Massey triple products). The codimension gives the simple connectivity property
\end{proof}

Here is the simplest of the above examples (see~\ref{quasihomogeneous}). Let $H_{ij}\subset\CC^6$ be the linear
subspace defined by $V(x_i,x_j)$. The complement
\[\CC^6\setminus (H_{12}\cup H_{23}\cup H_{34}\cup H_{45}\cup H_{56})\]
is simply connected and non-formal. The ideal of $H_{12}\cup H_{23}\cup H_{34}\cup H_{45}\cup H_{56}$ is 
\[ (x_1x_3x_4x_6,x_1x_3x_5,x_2x_3x_5,x_2x_4x_6,x_2x_4x_5). \]
Hence

\begin{example}
\label{ejemplo}
The Milnor fibre $F$ of the quasihomogeneous polynomial
\[f:\CC^{11}\to\CC\]
defined by
\[f(x_1,\cdots,x_6,y_1,\cdots,y_5):=y_1x_1x_3x_4x_6+y_2x_1x_3x_5+y_3x_2x_3x_5+y_4x_2x_4x_6+y_5x_2x_4x_5\]
is simply connected and non-formal. Its geometric monodromy is trivial. 

The homology of the Milnor fibre is computed from the homology of the complement of the subspace configuration. It can be done using Theorem A of~\cite{GM},~page~238, 
or using the formalism~\cite{DS}, concretely Hochster's~formula~(26). The result non-zerohomology groups are:
$H_0(F,\ZZ)=\ZZ$, $H_3(F,\ZZ)=\ZZ^5$, $H_4(F,\ZZ)=\ZZ^4$, $H_6(F,\ZZ)=\ZZ^3$, $H_7(F,\ZZ)=\ZZ^4$, $H_8(F,\ZZ)=\ZZ$.
\end{example}

\end{document}